\documentclass[10pt]{article}
\usepackage{amsfonts}
\usepackage[leqno]{amsmath}
\usepackage{graphicx}
\usepackage{latexsym}
\usepackage{amsmath,amsfonts,amssymb,amsthm,mathrsfs,euscript,makeidx,color}
\usepackage{enumerate}
\allowdisplaybreaks[3]
\usepackage{multirow}
\usepackage{indentfirst}

\def\sqr#1#2{{\vcenter{\vbox{\hrule height.#2pt
				\hbox{\vrule width.#2pt height#1pt \kern#1pt \vrule width.#2pt}
				\hrule height.#2pt}}}}

\def\3n{\negthinspace \negthinspace \negthinspace }
\def\2n{\negthinspace \negthinspace }
\def\1n{\negthinspace }
\def\bel{\begin{equation}\label}
	\def\eel{\end{equation}}

\def\dbE{\mathbb{E}}
\def\dbF{\mathbb{F}}

\def\dbH{\mathbb{H}}

\def\dbP{\mathbb{P}}

\def\dbR{\mathbb{R}}
\def\dbS{\mathbb{S}}

\def\sD{\mathscr{D}}

\def\sR{\mathscr{R}}

\def\sU{\mathscr{U}}

\def\={\buildrel \triangle \over =}

\def\ds{\displaystyle}

\def\ns{\noalign{\ss}}
\def\a{\alpha}

\def\d{\delta}

\def\l{\lambda}

\def\si{\sigma}

\def\f{\varphi}

\def\i{\infty}
\def\Th{\Theta}
\def\L{\Lambda}

\def\O{\Omega}

\def\cD{{\cal D}}

\def\cF{{\cal F}}

\def\cK{{\cal K}}

\def\cM{{\cal M}}
\def\cN{{\cal N}}

\def\cU{{\cal U}}

\def\bel{\begin{equation}\label}
	\def\ee{\end{equation}}
\def\bt{\begin{theorem}\label}
	\def\et{\end{theorem}}
\def\bc{\begin{corollary}\label}
	\def\ec{\end{corollary}}
\def\bl{\begin{lemma}\label}
	\def\el{\end{lemma}}
\def\bp{\begin{proposition}\label}
	\def\ep{\end{proposition}}
\def\bas{\begin{assumption}}
	\def\eas{\end{assumption}}

\def\ba{\begin{array}}
	\def\ea{\end{array}}

\def\BS{{\bf S}}

\def\Th{\Theta}
\def\L{\Lambda}

\def\O{\Omega}

\def\no{\noindent}

\def\ss{\smallskip}
\def\ms{\medskip}

\def\q{\quad}
\def\qq{\qquad}
\def\hb{\hbox}

\def\liminf{\mathop{\underline{\rm lim}}}

\def\lan{{\langle}}
\def\ran{{\rangle}}

\def\esssup{\mathop{\rm esssup}}

\def\wt{\widetilde}

\def\cd{\cdot}
\def\cds{\cdots}

\def\les{\leqslant}
\def\ges{\geqslant}

\def\({\Big (}
\def\){\Big )}
\def\[{\Big[}
\def\]{\Big]}
\def\lan{\langle}
\def\ran{\rangle}

\def\bde{\begin{definition}\label}
	\def\ede{\end{definition}}

\def\rf{\eqref}

\def\square#1{\vbox{\hrule\hbox{\vrule height#1%
				\kern#1\vrule}\hrule}}
\def\rectangle#1#2{\vbox{\hrule\hbox{\vrule height#1%
				\kern#2\vrule}\hrule}}

	\font\tenbb=msbm10 \font\sevenbb=msbm7 \font\fivebb=msbm5
	
	\newfam\bbfam
	\scriptscriptfont\bbfam=\fivebb \textfont\bbfam=\tenbb
	\scriptfont\bbfam=\sevenbb

	\newtheorem{theorem}{Theorem}[section]
	\newtheorem{corollary}[theorem]{Corollary}
	
	\newtheorem{lemma}[theorem]{Lemma}
	\newtheorem{proposition}[theorem]{Proposition}

	\theoremstyle{definition}
	\newtheorem{definition}[theorem]{Definition}

	\newtheorem{remark}[theorem]{Remark}

	\newtheorem{example}{Example}[section]

	\makeatletter
	
	\@addtoreset{equation}{section}
	\makeatother
	
	\usepackage{xcolor}
	\makeatletter
	
	\input pdf-trans
	\newbox\qbox
	\def\usecolor#1{\csname\string\color@#1\endcsname\space}
	\newcommand\bordercolor[1]{\colsplit{1}{#1}}
	\newcommand\fillcolor[1]{\colsplit{0}{#1}}
	\newcommand\outline[1]{\leavevmode%
		\def\maltext{#1}%
		\setbox\qbox=\hbox{\maltext}%
		\boxgs{Q q 2 Tr \thickness\space w \fillcol\space \bordercol\space}{}%
		\copy\qbox%
	}
	\makeatother
	\newcommand\colsplit[2]{\colorlet{tmpcolor}{#2}\edef\tmp{\usecolor{tmpcolor}}%
		\def\tmpB{}\expandafter\colsplithelp\tmp\relax%
		\ifnum0=#1\relax\edef\fillcol{\tmpB}\else\edef\bordercol{\tmpC}\fi}
	\def\colsplithelp#1#2 #3\relax{%
		\edef\tmpB{\tmpB#1#2 }%
		\ifnum `#1>`9\relax\def\tmpC{#3}\else\colsplithelp#3\relax\fi
	}
	\bordercolor{black}
	\fillcolor{white}
	\def\thickness{.3}

\begin{document}

       \title{\bf Infinite Horizon   Mean-Field Linear-Quadratic Optimal Control Problems with Switching and Indefinite-Weighted
			Costs}
		
		\author{Hongwei Mei\footnote{ Department of Mathematics and Statistics, Texas Tech University, Lubbock, TX 79409, USA; email: {\tt hongwei.mei@ttu.edu}. This author is partially supported by Simons Travel Grant MP-TSM-00002835.},~~~Rui Wang\footnote{ Department of Mathematics and Statistics, Texas Tech University, Lubbock, TX 79409, USA; email: {\tt rui-math.wang@ttu.edu}},~~~Qingmeng Wei\footnote{School of Mathematics and Statistics, Northeast Normal University, Changchun 130024, China; email: {\tt weiqm100@nenu. edu.cn}. This  author is supported in part by   NSF of Jilin Province for Outstanding Young Talents (20230101365JC) and  NSF of   China (12371443), the National Key R\&D Program
of China (2023YFA1009002), the Changbai Talent Program of Jilin Province.} ~~~\text{and}~~~
			Jiongmin Yong\footnote{Department of Mathematics, University of Central Florida, Orlando, FL 32816, USA; email: {\tt jiongmin.yong@ucf.edu}. This author is supported in part by NSF Grant DMS-2305475.}  }
		
		\maketitle
		
		\no\bf Abstract: \rm  This paper is concerned with an infinite horizon  stochastic linear quadratic (LQ, for short) optimal control
		problems   with conditional mean-field terms in a switching environment. Different from \cite{Mei-Wei-Yong-2025}, the cost functionals do not have positive-definite weights here. %To work with the conditional mean-field term, we adopted the orthogonal decomposition method developed in \cite{Mei-Wei-Yong-2023}.
		When the problems are merely finite, we construct a sequence of asymptotic optimal controls and derive their closed-loop representations. For the solvability,  an equivalence result between  the open-loop and  closed-loop cases is established  through algebraic Riccati  equations and infinite horizon backward stochastic differential equations. It can be seen that the research  in \cite{Mei-Wei-Yong-2025} with positive-definite weights is a special case of the current paper.

		\ms
		
		\no\bf Keywords: \rm  Linear-quadratic optimization problem, infinite horizon,  conditional mean-field, Markov switching,
		
		\ms
		
		\no\bf AMS Mathematics Subject Classification. \rm 93E20, 49N10, 60F17.

\section{Introduction}

Since the    works of Kushner \cite{Kushner-1962} and Wonham \cite{Wonham-1968},
   the stochastic linear quadratic (LQ, for short)  control problems have been   investigated extensively in the literature.
Here we only highlight some studies related to indefinite weighted costs or random coefficients.
For stochastic LQ problems, Chen--Li--Zhou \cite{Chen-Li-Zhou-1998} found that  the weighting matrix of the control in the cost functional does not have to be positive definite, or even could be negative definite to some extent.
 This kind of problems has been termed to be {\it indefinite LQ problems}, for which we also refer to the study in  \cite{Rami-Moore-Zhou-2001, Rami-Moore-Zhou-2000, Chen-Zhou-2000, Du-2015, Lim-Zhou-1999}, etc.
Under the framework of open-loop and closed-loop solvability,  the    indefinite LQ problems with deterministic coefficients have been investigated  in \cite{Mei-Wei-Yong-2021, Sun-2016, Sun-Yong-2020a}, etc.

 For stochastic LQ problem with random  coefficients, Chen--Yong \cite{Chen-Yong-2000, Chen-Yong-2001} studied the local solvability of backward stochastic differential Riccati equations.
 Tang \cite{Tang-2003, Tang-2015} solved the case of random coefficients with degenerate (positive semi-definite) control weight in the cost functional.
 %For the classical theory of deterministic LQ problems, see Lee--Markus \cite{Lee-Markus-1967}, Willems \cite{Willems-1971}, Anderson--Moore \cite{Anderson-Moore-1971}, Wonham \cite{Wonham-1979}. See also Bernhard \cite{Bernhard-1979}, Zhang \cite{Zhang-2005}, Delfour \cite{Delfour-2007}, and Delfour-Sbarba \cite{Delfour-Sbarba-2009} for a zero-sum differential game version.
  %See also McLane \cite{McLane-1971}, Davis \cite{Davis-1977}, Bensoussan \cite{Bensoussan-1981}, and so on.
Later, Sun--Xiong--Yong \cite{Sun-Xiong-Yong-2021}    solved the general case of indefinite LQ problem with random coefficients. See L\"u--Wang  \cite{Lv-Wang-2023} for an infinite-dimensional version.
With the development of the  mean-field theory, Yong \cite{Yong-2013} studied the stochastic LQ problem with mean-field, and then followed by  \cite{Huang-Li-Yong-2015, Li-Sun-Yong-2016, Sun-2017, Wei-Yong-Yu-2019, Sun-Yong-2020b, Li-Shi-Yong-2021},  etc. In a switching environment,   the mean-field LQ problems   are also studied, see, for example,   \cite{Pham-2016, Wen-Li-2023, Zhang-Li-Xiong-2018}, etc.  Recently,   the stochastic LQ problem  with coefficients being adapted to a martingale and with conditional mean-field interaction was   introduced and  studied by Mei--Wei--Yong \cite{Mei-Wei-Yong-2024} in   the framework of finite horizon. For the infinite horizon case,  the interested readers can refer to  \cite{Mei-Wei-Yong-2025}, where the positive definiteness  condition is assumed for the weights of the cost functional. This work focuses on  the  indefinite weights case of   infinite horizon mean-field LQ problem  in a switching environment, which can be regarded as  a continuation of  \cite{Mei-Wei-Yong-2024, Mei-Wei-Yong-2025}.

In details, we work on  a complete filtered probability space   $(\O,\cF,\dbF,\dbP)$ on which a one-dimensional standard Brownian motion $W(\cd)$ and a Markov chain $\a(\cd)$ are defined respectively. The Markov chain $\a(\cd)$ is equipped with a finite state space $\cM=\{1,\cds,m_0\}$ and a generator $\L=(\l_{\iota\jmath})_{m_0\times m_0}$. Moreover, we assume that $W(\cd)$ and $\a(\cd)$ are independent. Denoting by $\dbF^W=\{\cF_t^W\}_{t\ges 0}$ and $\dbF^\a=\{\cF_t^\a\}_{t\ges 0}$  the natural filtration of $W(\cd)$ and $\a(\cd)$, augmented by $\cN$ (all the $\dbP$-null sets), respectively. We introduce  the filtration $\dbF =\{\cF_t \}_{t\ges 0} $ with  $\cF_t :=\cF_t^W\vee\cF_t^\a\vee\cN$, $t\ges 0$.

Consider the following $n$-dimensional controlled mean-field stochastic differential equation (MF-SDE, for short) with regime switching (governed by the Markov chain $\a(\cd)$):
\bel{SDE-nonhomo}\left\{\2n\ba{ll}
\ns\ds \!\! dX(t)\!=\!\big\{A(\a(t))X(t)\!+\!\bar A(\a(t))\dbE_t^\a[X(t)]\!+\!B(\a(t))u(t)\!+\!\bar B(\a(t))\dbE_t^\a[u(t)]\!+\!b(t)\big\}dt\\[1mm]
\ns\ds\qq\q+\big\{C(\a(t))X(t)\!+\!\bar C(\a(t))\dbE_t^\a[X(t)]\!+\!D(\a(t))u(t)\!+\!\bar D(\a(t))\dbE_t^\a[u(t)]\!+\!\si(t)\big\}dW(t),\q t\1n\ges\1n s,\\
\ns\ds\!\!  X(s)=x\in\dbR^n,\qq \a(s)=\iota\in\cM,\ea\right.\ee
and the following cost functional
\bel{cost-non}\ba{ll}
\ns\ds J^\infty(s,\iota,x ;u(\cd)) =\dbE\int_s^\infty f\big(t,\a(t),X(t),\dbE_t^\a[X(t)],u(t),\dbE_t^\a[u(t)]\big)dt,\ea\ee
where  $\dbE^\alpha_t[\cd ]$ denotes the conditional expectation with respect to $\cF_t^\a$, and for $(t,\iota ,x,\bar x,u,\bar u)\in [0,\i)\times \cM\times \dbR^n\times\dbR^n\times \dbR^m\times \dbR^m$,
\bel{f-F}\ba{ll}
\ds f(t,\iota ,x,\bar x,u,\bar u):=\frac12\[\lan Q(\iota)x,x\ran+2\lan S(\iota)x,u\ran+\lan R(\iota)u,u\ran +\lan\bar Q(\iota)\bar x,\bar x\ran\1n+2\lan\bar S(\iota) \bar x ,\bar u\ran\1n+\1n\lan\bar R(\iota)\bar u,\bar u\ran\\ [1mm]
\ns\ds\qq\qq\qq\qq\qq+2\lan q(t),x\ran+2\lan\bar q(t),\bar x\ran+2\lan r(t),u\ran+2\lan\bar r(t),\bar u\ran\] . \ea\ee
In the above, all the coefficient matrices  of the state equation \rf{SDE-nonhomo} and the quadratic weight matrices  of the cost functional \rf{cost-non} are  deterministic maps defined on $\cM$; the nonhomogeneous terms $b(\cd),\si(\cd)$ of the state equation \rf{SDE-nonhomo} and the linear weight  terms $q(\cd),r(\cd),\bar q(\cd),\bar r(\cd)$ of the cost functional \rf{cost-non} are some integrable stochastic processes on $[0,\i)$.  The specific conditions will be presented in  Section 2.
We note that,  $\a(\cd)$ being a Markov chain leads the coefficients (such as $A(\a(\cd))$) and weight functions (such as $Q(\a(\cd))$) to be  random.

Introducing the following  set of {\it admissible controls}
\bel{set-Ad}\sU^{s,\iota,x}_{ad}[s,\i):=\Big\{u(\cd)\in L^2_\dbF(s,\i;\dbR^m)\bigm|X(\cd\,;s,\imath,x,u(\cd))\in L^2_\dbF(s,\i;\dbR^n)\Big\},\ee
we formulate a {\it mean-field linear-quadratic  optimal control problem}  as follows.

\ss
\no {\bf Problem (MF-LQ)$^\i$} For any $(s,\iota,x)\in [0,\i)\times\cM\times\dbR^n$, minimize the cost functional \rf{cost-non} subject to the state equation \rf{SDE-nonhomo} over $\sU^{s,\iota,x}_{ad}[s,\i)$.

\ss

 In this paper, we pay attention to    Problem (MF-LQ)$^\i$   with not  necessarily positive-definite weight  costs, unlike Assumption (A3) in \cite{Mei-Wei-Yong-2025}. Here we still adopt the orthogonal decomposition method introduced firstly by \cite{Mei-Wei-Yong-2024} to  deal with the conditional mean-field terms in \rf{SDE-nonhomo} and  \rf{cost-non}.
 There are  some interesting phenomenons  appearing, which are presented in details as follows.

We notice that the result in \cite{Mei-Wei-Yong-2025} can be extended to the case when the cost functional is strictly convex without essential difficulties. Based on this, we first   focus on  the case when  Problem (MF-LQ)$^\i$  is merely finite. The lack of strict convexity  inspires us to   regularize   Problem (MF-LQ)$^\i$ by introducing a
 regularized   cost functional,  denoted by  Problem (MF-LQ)$^\i_\d$ with  $ \d>0 $. Note that, the cost functional of  Problem (MF-LQ)$^\i_\d$  is strictly convex, so that we can apply the results   in \cite{Mei-Wei-Yong-2025} to
get the closed-loop optimal control of  Problem (MF-LQ)$^\i_\d$  explicitly through a system of algebraic Riccati equations (AREs, for short) and backward stochastic differential equations (BSDEs, for short)  with the coefficients depending on the Markov chain.
It turns out that   the optimal controls of  Problem (MF-LQ)$^\i_\d$  provide     a
sequence of asymptotic optimal controls for Problem (MF-LQ)$^\i$.
Meanwhile, we point out that the optimal values of   Problem (MF-LQ)$^\i_\d$ will converge to   the optimal value of  Problem (MF-LQ)$^\i $, but not necessarily for the  optimal controls. An example is given to illustrate this important phenomenon. To tackle this insufficiency,  the solvability of  Problem (MF-LQ)$^\i$ is investigated. An  equivalent characterization between   the open-loop and closed-loop solvability  is given through the solvability of  a system of AREs and infinite horizon BSDEs with the coefficients depending on the Markov chain. Moreover,   the explicit representation of the optimal control     is derived in this case.

The rest of the paper is arranged as follows. First, some preliminary results from \cite{Mei-Wei-Yong-2025}  are recalled  in section \ref{sec:pre}. Then we carry out the study of   Problem (MF-LQ)$^\i$ when it is merely finite in Section \ref{sec:fin}.  Section 4 is about  the research of    the solvability of Problem (MF-LQ)$^\i$.    An  equivalent characterization between   the open-loop and closed-loop solvability  is  concluded. Finally some concluding remarks are made in section \ref{sec:con}.
		
\section{Preliminaries}\label{sec:pre}
	
In this section, we  recall some necessary  preliminary results  from \cite{Mei-Wei-Yong-2025}. For any Euclidean space $\dbH$ and $T\in(0,\i]$, we write
\begin{align*}
&L^2_{\cF_s}(\O;\dbH)=\Big\{\xi :\O\to\dbH\bigm|\xi \hb{ is $\cF_s $-measurable, and }\dbE|\xi|^2<\i\Big\},\\
& L_{\cF_s^\a}^2(\O;\dbH)=\Big\{\xi \in L^2_{\cF_s}(\O;\dbH)\bigm|\xi \hb{ is $\cF_s^\a$-measurable}\Big\},\\
&L^2_{\dbF}(s,T;\dbH)=\Big\{\f:[s,T]\times\O\to\dbH\bigm|\f(\cd)\hb{ is $\dbF$-progressively measurable, }\dbE\int_s^T|\f(t)|^2dt <\infty\Big\},\\
& L^2_{\dbF^\a}(s,T;\dbH)=\Big\{\f(\cd)\in L^2_\dbF(s,T;\dbH)  \bigm|  \f(\cd) \hb{ is $\dbF^\a$-progressively measurable}\Big\},\\
& L^2_{\dbF^\a_-}(s,T;\dbH)=\Big\{\f(\cd)\in L^2_\dbF(s,T;\dbH)  \bigm|  \f(\cd) \hb{ is $\dbF^\a$-predictable}\Big\},\\
&L^\i_{\dbF}(s,T;\dbH)=\Big\{\f:[s,T]\times\O\to\dbH\bigm|\f(\cd)\hb{ is $\dbF$-progressively measurable, }\esssup_{t\in[s,T]}\|\f(t,\cd)\|_\i <\infty\Big\},\\
& L^\i_{\dbF^\a}(s,T;\dbH)=\Big\{\f(\cd)\in L^\infty_{\dbF}(s,T;\dbH)\bigm| \f(\cd)\hb{ is $\dbF^\a$-progressively measurable}\Big\},\\
&\sD=\Big\{(s,\iota,x )\bigm|s\in[0,\i),\ \iota\in\cM,\ x \in L^2_{\cF_s}(\O;\dbR^n)\Big\}, \hb{whose element is called    an {\it admissible initial triple}}.\end{align*}
 For $0\les s <T<\i$, let
$$\sU[s,\i):=L^2_\dbF(s,\i;\dbR^m)\text{ and }\sU[s,T]:=L^2_\dbF(s,T;\dbR^m),$$
which is the set  of 	 {\it feasible controls} on $[s,\i)$ and $[s,T]$, respectively.  For MF-SDE \eqref{SDE-nonhomo} and the cost functional \eqref{cost-non}, we assume  the following throughout the paper.
\medskip

\no{\bf (A1).} The coefficients satisfy
\begin{align*}
 &A(\cd),\bar A(\cd),C(\cd),\bar C(\cd):\cM\to\dbR^{n\times n}, ~B(\cd),\bar B(\cd),D(\cd),\bar D(\cd):\cM\mapsto\dbR^{n\times m},\\
 &b(\cd),\sigma(\cd),~q(\cd)\in L^2_\dbF(0,\infty;\dbR^n),~\bar q(\cd)\in L^2_{\dbF^\a}(0,\i;\dbR^n),~r(\cd)\in L^2_\dbF(0,\i;\dbR^m),~ \bar r(\cd)\in L^2_{\dbF^\a}(0,\i;\dbR^m).
 \end{align*}

Obviously, under {\bf(A1)}, for any $T>0$, any admissible initial triple $(s,\iota,x)\in\sD$ and any feasible control $u(\cd)\in\sU[s,\i)$, the state equation \eqref{SDE-nonhomo} admits a unique solution $X(\cd)=X(\cd\,;s,\iota,x,u(\cd))\in L^2_\dbF(s,T;\dbR^n)$ .
This cannot ensure  $J^\i(s,\imath,x;u(\cd))$  to be well-defined,  so the introduction of the set \eqref{set-Ad} is necessary.
Note that such a set of admissible controls depends on the initial triple $(s,\iota,x)\in\sD$. In the future, we will derive an equivalent characterization that is initial-independent.

\subsection{Orthogonal decomposition of Problem (MF-LQ)$^\i$}
Following the idea of orthogonal decomposition introduced in \cite{Mei-Wei-Yong-2025}, this subsection introduces a new formulation of Problem (MF-LQ)$^\i$.
Recall  the orthogonal decomposition map $\Pi:L^2_{\dbF}(s,T;\dbH)\mapsto L^2_{\dbF^\a}(s,T;\dbH)$ defined by
\begin{equation*}
\Pi[v](t):=\dbE[v(t)|\cF^\a_t]\text{ for any } v(\cd)\in L^2_{\dbF}(s,T;\dbH).
\end{equation*}	
It was proved in \cite{Mei-Wei-Yong-2025}	that $\Pi$ induces the following orthogonal decomposition for $L^2_{\dbF}(s,T;\dbH)$ by	%
\begin{equation*}
L^2_{\dbF}(s,T;\dbH)=L^2_{\dbF^\alpha}(s,T;\dbH)^\perp\oplus L^2_{\dbF^\alpha}(s,T;\dbH),
\end{equation*}
where
\begin{equation*}
L^2_{\dbF^\a}(s,T;\dbH)^\perp:=\Big\{v(\cd)\in L^2_{\dbF}(s,T;\dbH)\bigm|\dbE\int_s^T \lan v(t),\bar v(t)\ran dt=0,\q\forall\bar v(\cd)\in L^2_{\dbF^\a}(s,T;\dbH)\Big\}.
\end{equation*}
Let $\Pi_1:=I-\Pi$ and $\Pi_2:=\Pi$. We then apply such an orthogonal decomposition on  Problem (MF-LQ)$^\i$. Utilizing the orthogonality, it can be seen that the solution $X$ of \eqref{SDE-nonhomo} admits the following decomposition
\begin{equation*}
X(\cd)=X_1(\cd)\oplus X_2(\cd)
\end{equation*}
satisfying
\begin{align}\label{SDE1}\begin{cases}
& \!\!\! \!\!\!dX_1(t)=[A_1(\a(t)) X_1(t)+B_1(\a(t))u_1(t)+b_1(t)]dt\\
&\!\!\! \!\!\!\qq+[C_1(\a(t)) X_1(t)+C_2(\a(t)) X_2+D_1(\a(t))u_1(t)+D_2(\a(t))u_2(t)+\sigma(t)]dW(t),\\
&\!\!\! \!\!\! dX_2(t)=[A_2(\a(t))X_2(t)+B_2(\a(t))u_2(t)+b_2(t)]dt, \q t\in[s,\i),\\
& \!\!\! \!\!\! X_1(s)=x_1,\q X_2(s)=x_2, \q \a(s)=\iota,
\end{cases}\end{align}
and the cost function in \eqref{cost-non} reduces to
\begin{align}\label{cost1}
&J^\infty(s,\iota, x_1\oplus x_2;u_1(\cd)\oplus u_2(\cd)):=J^\infty(s,\iota, x ;u(\cd))\\
&\q=\sum_{i=1}^2\dbE  \int_s^\infty\[\lan  Q_i(\a(t))X_i(t),X_i(t)\ran\! +\!2\lan S_i(\a(t))X_i(t),u_i(t)\ran\!+\!\lan R_i(\a(t))u_i(t),u_i(t)\ran\nonumber\\
&\qq\qq\qq\q +2\lan  q_i(t),X_i(t)\ran\!+2\lan
r_i(t),u_i(t)\ran\]dt.\nonumber
\end{align}
Here, for $i\in \cM$, $ $
\begin{align*}\begin{cases}	
& \!\!\! \!\!\! \Upsilon_1(i):=\Upsilon(i),\q\Upsilon_2(i):=\Upsilon(i)+\bar \Upsilon(i),  \hb{ for }\Upsilon=A,B,C,D,Q,S,R,\\
&\!\!\! \!\!\! q_1(t)=q(t)-\Pi[q(t)],\q q_2(t)=\Pi[q(t)+\bar q(t)],\q  r_1(t)=r(t)-\Pi[r(t)],\q r_2(t)=\Pi[r(t)+\bar r(t)],\\
&\!\!\! \!\!\! x_1=\Pi_1[x],~x_2=\Pi_2[x].
\end{cases}\end{align*}

 After the decomposition, the set of admissible controls can be rewritten as
\begin{align*}\sU^{s,\imath,x}_{ad}[s,\i)&=\Big\{u(\cd)=u_1(\cd)\oplus u_2(\cd)\in\sU[s,\i)\bigm|\\
&\qq\qq X_1(\cd\,;s,\imath,x,u(\cd))\in L^2_{\dbF^\a}(s,\i;\dbR^n)^\perp,X_2(\cd\,;s,\imath,x,u(\cd))\in L^2_{\dbF^\a}(s,\i;\dbR^n)\Big\}.
\end{align*}
Then  MF-SDE \eqref{SDE1} together with the cost functional \eqref{cost1} is an equivalent formulation of  Problem (MF-LQ)$^\i$. In the sequel, we always focus on such an equivalent formulation only.

When the non-homogeneous terms in \eqref{SDE1} and  the linear weight terms in
\eqref{cost1}  are all 0, we have
\begin{align}\label{SDE10}\begin{cases}
&  \!\!\! \!\!\!dX_1(t)=[A_1(\a(t)) X_1(t)+B_1(\a(t))u_1(t)]dt\\
& \!\!\! \!\!\!\qq+[C_1(\a(t)) X_1(t)+C_2(\a(t)) X_2+D_1(\a(t))u_1(t)+D_2(\a(t))u_2(t)]dW(t),\\
&  \!\!\! \!\!\!dX_2(t)=[A_2(\a(t))X_2(t)+B_2(\a(t))u_2(t)]dt, \q t\in[s,\i),\\
& \!\!\! \!\!\! X_1(s)=x_1,\q X_2(s)=x_2, \q \a(s)=\iota,
\end{cases}\end{align}
and the cost function in \eqref{cost-non} reduces to
\begin{align}\label{cost10}
&J^{0,\infty}(s,\iota, x_1,x_2;u_1(\cd),u_2(\cd))\\
&\q=\sum_{i=1}^2\dbE  \int_s^\infty\[\lan  Q_i(\a(t))X_i(t),X_i(t)\ran\! +\!2\lan S_i(\a(t))X_i(t),u_i(t)\ran\!+\!\lan R_i(\a(t))u_i(t),u_i(t)\]dt.\nonumber
\end{align}
Then we denote  by  Problem (MF-LQ)$_0^\i$  the LQ problem associated with the above SDE \eqref{SDE10}  together with cost functional \eqref{cost10}.

\subsection{Martingale measure of $\a(\cd)$}
This subsection recalls a martingale measure constructed from  the Markov chain  $\a(\cd)$.
For $\imath\ne\jmath$,  define
$$\ba{ll}
\ns\ds\wt M_{\imath\jmath}(t):=\sum_{0<s\les t}{\bf1}_{[\a(s_-)=\imath]}{\bf1}_{[\a(s)=\jmath]}\equiv\hb{accumulative jump number from $\imath$ to $\jmath$ %\ne\imath$
in $(0,t]$},\\
\ns\ds\lan\wt M_{\imath\jmath}\ran(t):=\int_0^t
\l_{\imath\jmath}{\bf1}_{[\a(s_-)=\imath]}ds,%\\
%\ns\ds
\q M_{\imath\jmath}(t):=\wt M_{\imath\jmath}(t)-\lan\wt M_{\imath\jmath}\ran(t),\qq t\ges0.\ea$$
According to \cite{Rogers-Williams-2000}, p.35, (21.12) Lemma (see also \cite{Donnelly-Heunis-2012, Rolon Gutierrez-Nguyen-Yin-2024}), the above $M_{\imath\jmath}(\cd)$ is a purely discontinuous and square-integrable martingale (with respect to $\dbF^\a$). For convenience, we let
$$M_{\imath\imath}(t)=\wt M_{\imath\imath}(t)=\lan\wt M_{\imath\imath}\ran(t)=0,\qq t\ges0.$$
Then $\{M_{\imath\jmath}(\cd)\bigm|\imath,\jmath\in\cM\}$ is the {\it martingale measure} of Markov chain $\a(\cd)$.

\ms

To define the stochastic integral with respect to such a martingale measure, we need to introduce the following Hilbert spaces
\bel{MFspace}\left\{\ba{ll}
\ns\ds\!\!\! M^2_{\dbF_-}(s,T;\dbH)=\Big\{\f(\cd)=(\f(\cd\,,1),\cds,\f(\cd\,,m_0))\bigm|
\f(\cd\,,\imath)\hb{ is $\dbH$-valued and $\dbF$-predictable }\\
\ns\ds\qq\qq\qq\qq \hb{ with } \sum_{\iota\neq\jmath}\dbE\int_s^T|\f(r,\jmath)|^2d\wt M_{\imath\jmath}(r)%\l_{\imath\jmath}{\bf1}_{[\a(r)=\imath]}dr
<\i,\q\forall\jmath\in\cM\Big\},\\
\ns\ds\!\!\! M^2_{\dbF^\a_-}(s,T;\dbH)=\Big\{\f(\cd)\in M^2_{\dbF_-}(s,T;\dbH)\bigm|\f(\cd) \hb{ is $\dbF^\a$-predictable} \Big\}.\ea\right.\ee
Clearly, $ M^2_{\dbF^\a_-}(s,T;\dbH)\subset  M^2_{\dbF_-}(s,T;\dbH)$. For any $\varphi(\cdot)\in  M^2_{\dbF_-}(s,T;\dbH)$, define
$$\int_s^t\f(r)dM(r):=\sum_{\imath\neq\jmath}\int_s^t\f(r,\jmath){\bf1}_
{[\a(r^-)=\imath]}dM_{\imath\jmath}(r),$$
which  is a (local) martingale with  quadratic variation
$$\dbE\(\int_s^t\f(r)dM(r)\)^2=\dbE\int_s^t\sum_{\imath\neq\jmath}|
\f(r,\jmath)|^2\l_{\imath\jmath}{\bf1}_{[\a(r)=\imath]}dr.$$
For any map $\Sigma:\cM\to\dbH$, we set $\Lambda[ \Sigma](\imath):=\sum\limits_{\jmath\in\cM}\lambda_{\imath\jmath}\Sigma(\jmath)$.

\subsection{Stabilizability}
	
In this subsection, we will recall the results on the stabilizability from \cite{Mei-Wei-Yong-2025}. The following definition is required.
\begin{definition}\label{Def-stab-1} (1).  $(\Th_1(\cd),\Th_2(\cd)):\cM\mapsto\dbR^{m\times n}\times \dbR^{m\times n}$ is said to be an {\it $L^2$-stabilizer} for  the following system (with $\a(t)$ suppressed)
\begin{align}\label{SDE10Th}\begin{cases}
&  \!\!\! \!\!\! dX_1^0(t)=(A_1+B_1\Th_1) X_1^0(t)dt +[(C_1+D_1{\Th_1})X_1^0(t) +(C_2+D_2\Th_2) X_2^0(t) ]dW(t),\\
& \!\!\! \!\!\! dX_2^0(t)= (A_2+B_2\Th_2) X_2^0(t)dt,\q t\in[s,\i),\\
& \!\!\! \!\!\!  X_1^0(s)=x_1,\q X_2^0(s)=x_2, \q \a(s)=\iota 	
\end{cases}\end{align}
if it admits the solution  $(X_1^0(\cd),X_2^0(\cd))\in L^2_{\dbF^\a}(s,\i;\dbR^n)^\perp\times L^2_{\dbF^\a}(s,\i;\dbR^n)$. %In this case we write $(X_1^0(\cd),X_2^0(\cd))$ by $(X_1^0(\cd;s,\iota,x_1,\Th_1,\Th_2),X_2^0(\cd;s,\iota,x_2,\Th_2))$.
		
(2).   $(\Th_1(\cd),\Th_2(\cd)):\cM\mapsto\dbR^{m\times n}\times \dbR^{m\times n}$ is said to be a {\it  dissipative strategy} of system \eqref{SDE10Th}  if there exist $P_1,P_2:\cM\mapsto\dbS^n_{++}$ such that, for any $\jmath\in\cM$,
\begin{equation}\label{dissiptcre}
\L[P_k]+(A_k+B_k\Th_k)^\top P_k+P_k(A_k+B_k\Th_k)+(C_k+D_k\Th_k)^\top P_1(C_k+D_k\Th_k)<0,\q k=1,2.
\end{equation}

	\end{definition}	
		
 Denote system \eqref{SDE10Th} by $[A_1,A_2, C_1,C_2; B_1,B_2,D_1,D_2]$ and write the set of all possible stabilizers by
$\BS[A_1,A_2, C_1,C_2; B_1,B_2,D_1,D_2].$
Then we have the following proposition	taken from 	 \cite{Mei-Wei-Yong-2025} which presents the equivalence of stabilizer and dissipative strategy.

\begin{proposition}\label{lemmaequista} \sl For system $[A_1,A_2, C_1,C_2; B_1,B_2,D_1,D_2]$, $(\Th_1,\Th_2)$ is an $L^2$-stabilizer  if and only if  $(\Th_1,\Th_2)$ is a  dissipative strategy.

\end{proposition}

Therefore we introduce the following hypothesis.

\ms

\noindent {\bf (A2).} $\BS[A_1,A_2, C_1,C_2; B_1,B_2,D_1,D_2]\ne\varnothing$, or equivalently $(\widehat\Th_1(\cd),\widehat\Th_2(\cd))\in\BS[A_1,A_2, C_1,C_2; B_1,B_2,D_1,D_2]$.
 \ss

Next, picking any   $(\widehat\Th_1(\cd),\widehat\Th_2(\cd)) \in\BS[A_1,A_2, C_1,C_2; B_1,B_2,D_1,D_2]$, and $v(\cd)\equiv v_1(\cd)\oplus v_2(\cd)\in\sU[s,\i)$, we call them a {\it closed-loop strategy}. Take
\bel{outcome}u_k^{\widehat\Th_k,v_k}(t)\equiv u_k(t)=v_k(t)+\widehat\Th_k(\a(t))X_k(t),\qq t\ges0,\ee
which is called the {\it outcome} of $(\widehat\Th_1(\cd),\widehat\Th_2(\cd),v_1(\cd),v_2(\cd))$, where $X(\cd)\equiv X_1(\cd)\oplus X_2(\cd)$ is the state process \rf{SDE1}  corresponding to the control $u(\cd)=u_1(\cd)\oplus u_2(\cd)$. Any such a form control is called a {\it closed-loop control}.
	Recall the set $\sU^{s,\iota,x}_{ad}[s,\i)$, we have the
  following result, which  is from \cite{Mei-Wei-Yong-2025} (Proposition 4.5 therein).
\begin{proposition}\label{equivalence}  \sl Let {\bf(A2)} hold. %and $(\widehat\Th_1(\cd),\widehat\Th_2(\cd))\in\BS[A_1,A_2, C_1,C_2; B_1,B_2,D_1,D_2]$.
Then for any $(s,\iota,x)\in\sD$,
\begin{equation}\label{su}\sU^{s,\iota,x}_{ad}[s,\i)=\Big\{u_1^{\widehat\Th_1,v_1}(\cd)\oplus u_2^{\widehat\Th_2,v_2}(\cd)\bigm|v_1(\cd)\oplus v_2(\cd)\in\sU[s,\i)\Big\}.
\end{equation}

\end{proposition}
We note that the right-hand side in \eqref{su} is independent of the choice of $(s,\iota,x)\in\sD$.
Take a fixed $(\widehat\Th_1(\cd),\widehat\Th_2(\cd))\in\BS[A_1,A_2, C_1,C_2; B_1,B_2,D_1,D_2]$, for any $$u_k(t)= \widehat\Th_k(\a(t))X_k(t)+v_k(t),\q k=1,2,$$ we have
\begin{align}\label{costTh}
&J^\infty(s,\iota,x_1\oplus x_2; u_1(\cdot)\oplus u_2(\cdot))\\
%
%&\q=\sum_{i=1}^2\dbE  \int_s^\infty\[\lan  Q_i(\a(t))X_i(t),X_i(t)\ran\! +\!2\lan S_i(\a(t))X_i(t),u_i(t)\ran\!+\!\lan R_i(\a(t))u_i(t),u_i(t)\ran\nonumber\\
%%
%&\qq\qq +2\lan  q_i(t),X_i(t)\ran\!+2\lan r_i(t),u_i(t)\ran\]dt\nonumber\\
%%
%&\q=\sum_{i=1}^2\dbE  \int_s^\infty\[\lan  Q_i(\a(t))X_i(t),X_i(t)\ran\! +\!2\lan S_i(\a(t))X_i(t),\widehat \Th_i(\a(t))X_i(t)+v_i(t)\ran\nonumber\\
%%
%&\qq\qq+\!\lan R_i(\a(t))\widehat\Th_i(\a(t))X_i(t)+v_i(t),\widehat\Th_i(\a(t))X_i(t)+v_i(t)\ran\nonumber\\
%%
%&\qq\qq+2\lan  q_i(t),X_i(t)\ran+2\lan r_i(t),\widehat\Th_i(\a(t))X_i(t)+v_i(t)\ran\]dt\nonumber\\
%%
&\q=\sum_{i=1}^2\dbE  \int_s^\infty\[\lan  Q^{\widehat\Th_i}_i(\a(t))X_i(t),X_i(t)\ran\! +\!2\lan S^{\widehat\Th_i}_i(\a(t))X_i(t),v_i(t)\ran\!+\!\lan R^{\widehat\Th_i}_i(\a(t))v_i(t),v_i(t)\ran\nonumber\\
&\qq\qq+2\lan  q^{\widehat\Th_i}_i(t),X_i(t)\ran\!+2\lan r^{\widehat\Th_i}_i(t),v_i(t)\ran\]dt\nonumber\\
&=:\widehat J^\infty(s,\iota,x_1\oplus x_2; v_1(\cdot)\oplus v_2(\cdot)),\nonumber
\end{align}
where
\begin{align*}
& Q_i^{\widehat\Th_i}=Q_i+\widehat\Th_i^\top S_i+S_i^\top\widehat\Th_i+\widehat\Th_i^\top R_i\widehat\Th_i,\q  S_i^{\widehat\Th_i}=S_i+R_i\widehat\Th_i,\\
&q^{\widehat\Th_i}_i(t)=q_i(t)+\widehat\Th^\top_i(\a(t))r_i(t),\q r^{\widehat\Th_i}_i(t)=r_i(t).
\end{align*}
Moreover, \eqref{SDE1} becomes
\begin{align}\label{SDE2}\begin{cases}
& \!\!\! \!\!\!  dX_1(t)=[A_1^{\widehat\Th_1}(\a(t)) X_1(t)+B^{\widehat\Th_1}_1(\a(t))v_1(t)+b_1(t)]dt\\
&\qq+[C^{\widehat\Th_1}_1(\a(t)) X_1(t)+C^{\widehat\Th_2}_2(\a(t)) X_2+D^{\widehat\Th_1}_1(\a(t))v_1(t)+D^{\widehat\Th_2}_2(\a(t))v_2(t)+\sigma(t)]dW(t),\\
& \!\!\! \!\!\!  dX_2(t)=[A^{\widehat\Th_2}_2(\a(t))X_2(t)+B^{\widehat\Th_2}_2(\a(t))v_2(t)+b_2(t)]dt, \q t\in[s,\i),\\
& \!\!\! \!\!\!  X_1(s)=x_1,\q X_2(s)=x_2, \q \a(s)=\iota,
\end{cases}\end{align}
where
$
 A_i^{\widehat\Th_i}=A_i+B_i\widehat\Th_i,\  C_i^{\widehat\Th_i}=C_i+D_i\widehat\Th_i,\ B_i^{\widehat\Th_i}=B_i,\ D_i^{\widehat\Th_i}=D_i.
$
At the same time, we can notice that {\bf (A2)} can be equivalently represented by$$(0,0)\in\BS[A_1^{\widehat\Th_1},A_2^{\widehat\Th_2}, C_1^{\widehat\Th_1},C_2^{\widehat\Th_2}; B_1^{\widehat\Th_1},B_2^{\widehat\Th_2},D_1^{\widehat\Th_1},D_2^{\widehat\Th_2}].$$
Observed from the above, we can state the main problem in our paper.\medskip

\no {\bf Problem (MF-LQ)$^{\infty}$.}  Let {\bf(A2)} hold. For any $(s,\iota,x_1\oplus x_2)\in\sD$, find a $u^*_1(\cd)\oplus  u^*_2(\cd)\in\sU^{s,\iota,x}_{ad}[s,\infty)$ or equivalently $v^*_1(\cd)\oplus  v^*_2(\cd)\in\sU[s,\infty)$ with $u_i^*(\cd)=\widehat \Th_i(\a(\cd))X_i(\cd)+v_i^*(\cd)$ such that
\begin{align*}V^{\infty}(s,\iota,x_1\oplus x_2)&:=J^\infty(s,\iota,x_1\oplus x_2;  u_1^*(\cdot)\oplus  u^*_2(\cdot))\\&=\inf_{u_1(\cdot)\oplus u_2(\cdot)\in \cU^{s,\iota,x}_{ad}[s,\infty)}J^\infty(s,\iota,x_1\oplus x_2; u_1(\cdot)\oplus u_2(\cdot))\\
&=\widehat J^\infty(s,\iota,x_1\oplus x_2;  v_1^*(\cdot)\oplus  v^*_2(\cdot))\\
&=\inf_{v_1(\cdot)\oplus v_2(\cdot)\in \sU[s,\infty)}\widehat  J^\infty(s,\iota,x_1\oplus x_2; v_1(\cdot)\oplus v_2(\cdot)).\end{align*}

Observed from the above, it suffices to solve Problem  (MF-LQ)$^{\infty}$ under the assumption that $(\widehat\Th_1,\widehat \Th_2)=(0,0)$ or equivalently
\ms

{\bf (A2)$'$} $(0,0)\in\BS[A_1,A_2, C_1,C_2; B_1,B_2,D_1,D_2]$.
\ms

In this case, we have $\sU^{s,\iota,x}_{ad}[s,\infty)=\sU[s,\i)$. Otherwise, we can work with the optimization problem \eqref{costTh} subject to \eqref{SDE2}. Therefore, when deriving the main results of the paper, we will assume {\bf (A2)$'$} first. Then we will present a parallel result under {\bf (A2)}. Now we present the following definition about Problem (MF-LQ)$^{\infty}$.

\bde{finite} \rm (1) Problem (MF-LQ)$^\i$ is said to be {\it finite} at $(s,\imath,x)\in\sD$ if $\cU^{s,\iota,x}_{ad}[s,\infty)\ne\varnothing$, and
\bel{>-i}\inf_{u(\cd)\in\sU^{s,\iota,x}_{ad}[s,\i)}J^\i(s,\imath,x ;u(\cd))>-\i.\ee
If the above is true for any $(s,\imath,x )\in\sD$, we simply say that Problem (MF-LQ)$^\i$ is {\it finite}.

\rm (2) A control $u^*(\cd)\in\sU^{s,\iota,x}_{ad}[s,\i)$ is called an {\it open-loop} optimal control of Problem (MF-LQ)$^\i$ for the initial $(s,\iota,x )\in\sD$,  if
\begin{equation*}
J^\i(s,\imath,x ;u^*(\cd))\les J^\i(s,\imath,x ;u(\cd))\text{ for any }u(\cd)\in \sU^{s,\iota,x}_{ad}[s,\i).
\end{equation*}
In this case, Problem (MF-LQ)$^\i$ is called {\it open-loop solvable at }$(t ,\imath ,x)\in\sD$. If Problem (MF-LQ)$_\i$ is open-loop solvable at all $(t,\imath ,x)\in\sD$, then Problem (MF-LQ)$^\i$ is called open-loop solvable.

\rm (3)
A strategy  $u_1^{\Th^*_1,v^*_1}(\cd)\oplus u_2^{\Th^*_2,v^*_2}(\cd)$ with $(\Th^*_1(\cd),\Th^*_2(\cd))\in\BS[A_1,A_2, C_1,C_2; B_1,B_2,D_1,D_2]$ and $v^*_1(\cd)\oplus v^*_2(\cd)\in\sU[s,\infty)$  is said be {\it closed-loop optimal}  if
\bel{J<J}J^\i(s,\imath,x ;u_1^{\Th^*_1,v^*_1}(\cd)\oplus u_2^{\Th^*_2,v^*_2}(\cd))\les
J^\i(s,\imath,x ;u_1^{\Th_1,v_1}(\cd)\oplus u_2^{\Th_2,v_2}(\cd)),\ee
for any $(\Th_1,\Th_2)\in\BS[A_1,A_2, C_1,C_2; B_1,B_2,D_1,D_2]$, $v_1(\cd)\oplus v_2(\cd)\in\sU[s,\i)$ and $(s,\iota,x )\in\sD$. When this happens, we say that Problem (MF-LQ)$^\i$ is {\it closed-loop solvable}.

\ede

Our main goal in the paper is to solve Problem (MF-LQ)$^{\infty}$ without the positive-definiteness condition of the cost functional, which was assumed in \cite{Mei-Wei-Yong-2025}.

\section{Finiteness}\label{sec:fin}

In this section, we will focus on the case when Problem  (MF-LQ)$^{\infty}$ is finite only. In this scenario, there exists a sequence of open-loop controls that is asymptotic optimal while there might not exist any optimal control. Our main contribution in this section lies in deriving the explicit forms for the asymptotic optimal sequence by systems of AREs and BSDEs. To achieve this, we proceed with some parallel propositions to Proposition 5.1 and 5.2 in \cite{Sun-Yong-2018}.
The proofs  are  thus omitted here.
  We emphasize that if {\bf(A2)$'$} is assumed, it follows that
$$\sU^{s,\iota,x}_{ad}[s,\infty)=\sU[s,\i),\text{ for any $(s,\iota,x)\in\sD$}.$$
We also note that we may use the  decompositions
$x=x_1\oplus x_2,~ u(\cd)=u_1(\cd)\oplus u_2(\cd)$ in the sequel.

\begin{proposition}\label{IN-new-rep-cost} \sl
 Suppose {\bf (A1)} and {\bf (A2)$'$} hold. For any $(s,\iota,x)\in\sD$ and $u(\cd)\in\sU[s,\i)$,  there  exist  bounded self-adjoint linear
operators
\begin{align*}&
\cK_2(s,\iota): L^2_{\dbF^\alpha}(0,\infty,\dbR^m)^\perp\times L^2_{\dbF^\alpha}(0,\infty,\dbR^m)\mapsto L^2_{\dbF^\alpha}(0,\infty,\dbR^m)^\perp\times L^2_{\dbF^\alpha}(0,\infty,\dbR^m),
\\
&
\cK_1(s,\iota):L_{\cF^\alpha_s}(\dbR^m)^\perp\times L_{\cF^\alpha_s}(\dbR^m)\mapsto  L^2_{\dbF^\alpha}(0,\infty,\dbR^m)^\perp\times L^2_{\dbF^\alpha}(0,\infty,\dbR^m)\\
&\cK_0(s,\iota):L_{\cF^\alpha_s}(\dbR^m)^\perp\times L_{\cF^\alpha_s}(\dbR^m)\mapsto L_{\cF^\alpha_s}(\dbR^m)^\perp\times L_{\cF^\alpha_s}(\dbR^m),\\
&\theta(s,\iota)\in  L^2_{\dbF^\alpha}(0,\infty,\dbR^m)^\perp\times L^2_{\dbF^\alpha}(0,\infty,\dbR^m)\\
& \vartheta(s,\iota)\in L_{\cF^\alpha_s}(\dbR^n)^\perp\times L_{\cF^\alpha_s}(\dbR^n),\q c(s,\iota)\in\dbR
\end{align*}  such that
\begin{align}\label{new-rep-cost-4}& J^\i(s,\iota,x;u(\cd))=\lan \cK_2(s,\iota)u,u\ran+\lan \cK_1(s,\iota)x,u\ran+\lan \cK_0(s,\iota)x,x\ran+2 \lan u, \theta(s,\iota)\ran +2\lan x,\vartheta(s,\iota) \ran   + c(s,\iota),\\
	& J^{\i,0}(s,\iota,x;u(\cd))=\lan \cK_2(s,\iota)u,u\ran+\lan \cK_1(s,\iota)x,u\ran+\lan \cK_0(s,\iota)x,x\ran\nonumber
    \end{align}

% \bel{new-rep-cost-3}\ba{ll}
% \ns\ds J_\i(s,\iota,x_1\oplus x_2;u_1(\cd)\oplus u_2(\cd))\\
% \ns\ds=\lan \cM_{1}(s,\iota)u_1,u_1\ran+\lan \cM_2(s,\iota)u_2,u_2\ran+2\lan \cN_1(s,\iota)x_1,u_1\ran+2\lan \cN_2(s,\iota)x_2,u_2\ran\\
% 	%
% 	\ns\ds\qq+\lan \cK_1(s,\iota)x_1,x_1\ran+\lan \cK_2(s,\iota)x_2,x_2\ran+2 \lan u_1, \cL_1(s,\iota)\ran +2\lan u_2, \cL_2(s,\iota)\ran\\
% 	%
% 	\ns\ds\qq +2\lan x_1,\cY_1(s,\iota) \ran+\lan x_2,\cY_2(s,\iota)\ran]   + c(s,\iota),\\
% 	%
% 	\ns\ds J_\i^0(s,\iota,x_1\oplus x_2;u_1(\cd)\oplus u_2(\cd))\\
%     \ns\ds=\lan \cM_{1}(s,\iota)u_1,u_1\ran+\lan \cM_2(s,\iota)u_2,u_2\ran+2\lan \cN_1(s,\iota)x_1,u_1\ran+2\lan \cN_2(s,\iota)x_2,u_2\ran\\
% 	%
% 	\ns\ds\qq +\lan \cK_1(s,\iota)x_1,x_1\ran+\lan \cK_2(s,\iota)x_2,x_2\ran.
% 	\ea\ee
	%

\end{proposition}

\begin{proposition}\label{solvab} \sl Suppose {\bf (A1)} and {\bf (A2)$'$} hold.

\ms

{\rm (i)} Problem {\rm (MF-LQ)$^{\infty}$} is open-loop solvable at $(s,\iota,x)\in\sD$ if and only if $\cK_2(s,\iota)\ges 0$, $\theta(s,\iota)+\cK_1(s,\iota)x\in\sR(\cK_2(s,\iota))$. In this case, $u^*(\cd)=u_1^*(\cd)\oplus u_2^*(\cd)$ is the optimal control if and and if %
$$\cK_2(s,\iota)u^*(\cd)+\cK_1(s,\iota)x+\theta(s,\iota)=0.$$

{\rm (ii)} If Problem {\rm (MF-LQ)$^{\infty}$} is open-loop solvable, then so is Problem {\rm(MF-LQ)$^0_{\infty}$}. Moreover, $J^0_\infty(s,\iota,x_1\oplus x_2,u_1(\cd)\oplus u_2(\cd))$ is convex over  $u_1(\cd)\oplus u_2(\cd)\in L^2_{\dbF^\alpha}(0,\infty;\dbR^m)^\perp\times L^2_{\dbF^\alpha}(0,\infty;\dbR^m).$

\ms

{\rm (iii)} If Problem {\rm (MF-LQ)$_0^{\infty}$} is open-loop solvable, then there exists a linear operator $$U^*(s,\iota): L_{\cF^\alpha_s}(\dbR^n)^\perp\times L_{\cF^\alpha_s}(\dbR^n)\mapsto L^2_{\dbF^\a}(0,\infty;\dbR^{m})^\perp\times L^2_{\dbF^\a}(0,\infty;\dbR^{m})$$ such that $u^*(\cd)= [U^*(s,\iota)x](\cd) $
is the optimal open-loop control for Problem {\rm (MF-LQ)$_0^{\infty}$}.
\end{proposition}

To find the asymptotic optimal sequence of optimal controls, first  we notice that  from \eqref{new-rep-cost-4}, it must hold that $\cK_2(s,\iota)\ges 0$ when Problem (MF-LQ)$^{\infty}$ is finite, i.e.
    $J_\infty^0(s,\iota,x_1\oplus x_2, u_1(\cd)\oplus u_2(\cd))$  is convex over $u_1(\cd)\oplus u_2(\cd)\in\sU[s,\infty).$ To gain the strict convexity,
we propose the following regularized cost functionals with $\d>0$:
\begin{align}\label{hcost20-d}&J_{\d}^\infty(s,\iota,x_1\oplus x_2;u_1(\cd)\oplus u_2(\cd))
	=J^\infty(s,\iota,x_1\oplus x_2;u_1(\cd)\oplus u_2(\cd))+\d\dbE\int_s^\infty\sum_{i=1}^2|u_i(t)|^2dt,\\
	&J^{0,\infty}_{\d}(s,\iota,x_1\oplus x_2;u_1(\cd)\oplus u_2(\cd))
	=J^{0,\infty}(s,\iota,x_1\oplus x_2;u_1(\cd)\oplus u_2(\cd))+\d\dbE\int_s^\infty\sum_{i=1}^2|u_i(t)|^2dt.\nonumber\end{align}
Then,  a regularized control problem of  Problem (MF-LQ)$^\i$ can be formulated as follows.

\ss

\no {\bf Problem (MF-LQ)$_{\d}^\infty$:} Let {\bf (A2)$'$} hold. Given $( s,\iota,x_1\oplus x_2)\in\cD$, find a $ u_{1,\d}^*(\cd)\oplus u_{2,\d}^*(\cd)\in\sU[s,\infty)$ such that
\begin{align*}V_{\d}^\infty(s,\iota,x_1\oplus x_2)&:=J_{\d}^\infty(s,\iota,x_1\oplus x_2; u_{1,\d}^*(\cd)\oplus u_{2,\d}^*(\cd))\\
&=\inf_{u_1(\cdot)\oplus u_2(\cdot)\in \sU[s,\infty)}J_{\d}^\infty(s,\iota,x_1\oplus x_2; u_1(\cdot)\oplus  u_2(\cdot)).\end{align*}

Due to the strict convexity,  Problem (MF-LQ)$_{\d}^{\infty}$ admits a unique open-loop optimal control $u_{1,\d}^*(\cd)\oplus u_{2,\d}^*(\cd)\in\sU[s,\infty)$. It is natural to show such a sequence is asymptotic optimal for  Problem (MF-LQ)$^{\infty}$ in the following proposition.

\begin{proposition} \sl Suppose   {\bf(A2)$'$} hold and Problem {\rm(MF-LQ)$^{\infty}$} is finite.
	
{\rm (1).} It follows that \begin{align}\label{cJd0}&\lim_{\d\rightarrow 0}V_{\d}^\infty(s,\iota, x_1\oplus x_2)=V^\infty(s,\iota,x_1\oplus x_2).
\end{align}
	
{\rm (2).} The sequence of optimal controls $ u_{1,\d}^*(\cd)\oplus u_{2,\d}^*(\cd)\in\sU[s,\infty)$ for  Problem (MF-LQ)$_{\d}^{\infty}$, is asymptotic optimal for  Problem (MF-LQ)$^{\infty}$ as $\d\rightarrow 0^+$ at $(s,\iota,x)\in\sD$
	
\end{proposition}
\begin{proof} (1).
There exists a  sequence of $u_1^i(\cd)\oplus u_2^i(\cd)\in \sU^{s,\iota,x}_{ad}[s,\i)=\sU[s,\infty)$ such that
\begin{align*}
&V_{\d}^\infty(s,\iota, x_1\oplus x_2)\ges V^{\i}(s,\iota, x_1\oplus x_2) \\
&\q=\inf\limits_{u_1(\cd)\oplus u_2(\cd)\in\sU[s,\i)}J^\infty(s,\iota, x_1\oplus x_2; u_1(\cdot)\oplus  u_2(\cdot))\ges J^\infty(s,\iota,x_1\oplus x_2;u_1^i(\cdot)\oplus  u_2^i(\cdot))-\frac1i\\
%p
&\q= J_{\d}^\infty(s,\iota,x_1\oplus x_2;u_1^i(\cdot)\oplus  u_2^i(\cdot))-\d\dbE\int_s^\infty\big[|u_1^i(t)|^2+u_2^i(t)|^2\big]dt-\frac1i \\
&\q\ges  \inf\limits_{u_1(\cd)\oplus u_2(\cd)\in\sU[s,\i)}J_\d^\infty (s,\iota, x_1\oplus x_2; u_1(\cdot)\oplus  u_2(\cdot))-\d\dbE \int_s^\infty\big[|u_1^i(t)|^2+u_2^i(t)|^2\big]dt-\frac1i\\
&\q=V_{\d}^\infty(s,\iota, x_1\oplus x_2)-\d\dbE \int_s^\infty\big[|u_1^i(t)|^2+u_2^i(t)|^2\big]dt-\frac1i.
\end{align*}
Letting $\d\to 0$ and then $i\to\i$, we  get the desired result.\ss
	
(2). By \eqref{cJd0}, we have
$$\ba{ll}\ds\lim_{\d\rightarrow 0}J^\infty(s,\iota,x_1\oplus x_2;  u_{1,\d}^*(\cd)\oplus   u^*_{2,\d}(\cd))\les\lim_{\d\rightarrow 0}J_{\d}^\infty(s,\iota,x_1\oplus x_2;  u^*_{1,\d}(\cd)\oplus u_{2,\d}^*(\cd))\\
\ns\ds=\inf_{ u_1(\cdot)\oplus u_2(\cdot)\in \sU[s,\infty)}J^\infty(s,\iota,x_1\oplus x_2; u_1(\cdot)\oplus u_2(\cdot))>-\infty.\ea$$
This proves $ u_{1,\d}^*(\cd)\oplus u_{1,\d}^*(\cd)\in\sU[s,\infty)$  is asymptotic optimal for  Problem (MF-LQ)$^{\infty}$.
\end{proof}

We have constructed the required sequence of asymptotic optimal controls only while we have little knowledge about its form.  The following theorem  derives an explicit representation of $ u_{1,\d}^*(\cd)\oplus u_{2,\d}^*(\cd)\in\sU[s,\infty)$  using AREs and BSDEs.  We will omit the proof as it is  parallel to that of Theorem 5.1 in \cite{Mei-Wei-Yong-2025} under strict convexity instead of positive definiteness. 

\begin{theorem}\label{P-i-d-opti} \sl Suppose  {\bf (A2)$'$} hold and Problem {\rm (MF-LQ)$^{\infty}$} is finite. The following is true.\ss
	
{\rm (1)} The following algebraic equation
\begin{align}\label{BSAREPPP}\begin{cases}&\!\!\!\!\!\! \L[P^*_{i,\d}]+P^*_{i,\d}A_i+A_i^\top P^*_{i,\d}+C_i^\top P^*_{1,\d}C_i +Q_i\\
&\!\!\!\!\!\!\qq-(B_i^\top P^*_{i,\d}+ D_i^\top P^*_{1,\d} C_i+ S_i) ^\top( R_i+D_i^\top P^*_{1,\d}D_i+\d I)^{-1}(B_i^\top P^*_{i,\d}+ D_i^\top P^*_{1,\d} C_i+ S_i)=0,\\
&\!\!\!\!\!\!R_i+D_i^\top P^*_{1,\d}D_i\ges 0,\q i=1,2,\end{cases}\end{align}
admits a unique solution  $(P^*_{1,\d},P^*_{2,\d}):\cM\mapsto\dbS^n\times \dbS^n$ such that $(\Th^*_{1,\d},\Th^*_{2,\d})\in\BS[A_1,A_2, C_1,C_2; B_1,B_2,D_1,D_2]$ with
\bel{optimalThhomo}
\Th^*_{i,\d}:=-( R_i+D_i^\top P^*_{1,\d}D_i+\d I)^{-1}(B_i^\top P^*_{i,\d}+ D_i^\top P^*_{1,\d} C_i+ S_i),\q i=1,2.\\
\eel

{\rm(2)} There exist  $(\pi^*_{1,\d},\eta^*_{\d}, \nu^*_{1,\d})\in L^2_{\dbF^\a}(s,\infty;\dbR^n)^\perp\times  L_{\dbF}^2(s,\infty;\dbR^n)\times M_{\dbF^\a_-}^2(s,\infty;\dbR^n)^\perp$ and  $(\pi^*_{2,\d},\nu^*_{2,\d})\in L^2_{\dbF^\a}( s,\infty;\dbR^n) \times M_{\dbF^\a_-}^2(s,\infty;\dbR^n) $ satisfying the following infinite horizon BSDEs on $[s,\i)$
\bel{eqpiooo0}\left\{\ba{ll}
\ns\ds\!\!\! d\pi^*_{1,\d}(t)-\eta^*_{\d}(t)dW(t)-\nu^*_{1,\d}dM(t) \\
\ns\ds\q + \[(A_1^{\Th^*_{1,\d}})^\top\pi^*_{1,\d}\!+\!(C_1^{\Th^*_{1,\d}})^\top (P^*_{1,\d}\sigma_1\!+\!\Pi[\eta^*_{\d}])\!+\!P^*_{1,\d}b_1\!+\!q_1\!+\!{\Th^*_{1,\d}}^\top r_1\]dt=0,\\
\ns\ds\!\!\!\! d\pi^*_{2,\d}(t)\!-\!\nu^*_{2,\d}dM(t) \!+\! \[ (A_2^{\Th^*_{1,\d}})^\top\pi^*_{1,\d}\!+\!(C_2^{\Th^*_{2,\d}})^\top (P^*_{1,\d}\sigma_2\!+\!\Pi^\perp[\eta^*_{\d}])\!+\!P^*_{2,\d} b_2\!+\!q_2\!+\!{\Th^*_{2,\d}}^\top r_2\] dt=0,\\
\ns\ds\!\!\!\lim_{t\rightarrow\infty}\dbE|\pi^*_{i,\d}(t)|^2=0.\ea\right.\eel

{\rm (3)} The  optimal control for  Problem (MF-LQ)$_{\d}^{\infty}$ is
\bel{v-id}
u^*_{i,\d}(\cd) =\Th^*_{i,\d}X_i(\cd)-( R_i+D_i^\top P^*_{1,\d}D_i+\d I)^{-1} \[B_i^\top\pi^*_{i,\d}\!+\!D_i^\top\Pi_i[\eta^*_{\d}]\!+\!D_i^\top P^*_{1,\d}\sigma_i\!+\!r_i\],\q i=1,2,
\ee
which is asymptotic optimal for  Problem (MF-LQ)$^{\infty}$.
\end{theorem}

Without any essential difficulties, we can derive a similar result using {\bf (A2)}  instead of  {\bf (A2)$'$}.

\begin{proposition}\label{proA2}  \sl Suppose  {\bf (A2)} hold and Problem (MF-LQ)$^{\infty}$ is finite. The following is true.\ss
	
{\rm (1).} The following algebraic equation
\begin{align*}\begin{cases}
&\!\!\!\!\!\!0=\L[P^*_{i,\d}]+P^*_{i,\d}A_i+A_i^\top P^*_{i,\d}+C_i^\top P^*_{1,\d}C_i +Q_i\\
&\!\!\!\!\!\!\qq-(B_i^\top P^*_{i,\d}+D_i^\top P^*_{1,\d}C_i+ S_i)^\top( R_i+D_i^\top P^*_{1,\d}D_i+\d I)^{-1}(B_i^\top P^*_{i,\d}+D_i^\top P^*_{1,\d}C_i+ S_i)\\
&\!\!\!\!\!\!\qq+\d\widehat\Th_i^\top(R_i+D_i^\top P^*_{1,\d}D_i+\d I)^{-1}(B_i^\top P^*_{i,\d}+D_i^\top P^*_{1,\d}C_i+ S_i)\\
&\!\!\!\!\!\!\qq+\d(B_i^\top P^*_{i,\d}+D_i^\top P^*_{1,\d} C_i+S_i)^\top(R_i+D_i^\top P^*_{1,\d}D_i+\d I)^{-1}\widehat\Th_i\\
&\!\!\!\!\!\qq+\d\Th_i^\top(R_i+D_i^\top P^*_{1,\d}D_i)(R_i+D_i^\top P^*_{1,\d}D_i+\d I)^{-1}\widehat\Th_i\\
&\!\!\!\!\!\! R_i+D_i^\top P^*_{1,\d}D_i\ges 0,
\end{cases}
\end{align*}
admits a unique solution  $(P^*_{1,\d},P^*_{2,\d}):\cM\mapsto\dbS^n\times \dbS^n$ such that $(\widehat\Th_1+\Th^*_{1,\d},\widehat\Th_2+\Th^*_{2,\d})\in\BS[A_1,A_2, C_1,C_2; B_1,B_2,D_1,D_2]$ with
\begin{align*}
&\Th^*_{i,\d}:=-( R_i+D_i^\top P^*_{1,\d}D_i+\d I)^{-1}(B_i^\top P^*_{i,\d}+ D_i^\top P^*_{1,\d} C_i+ S_i)\\
&\qq-( R_i+D_i^\top P^*_{1,\d}D_i+\d I)^{-1}( R_i+D_i^\top P^*_{1,\d}D_i)\widehat\Th_i,\q i=1,2,
\end{align*}
where $(\widehat\Th_1,\widehat\Th_2)$  is the couple in  {\bf (A2)}.
	
{\rm(2).} There exist  $(\pi^*_{1,\d},\eta^*_{\d}, \nu^*_{1,\d})\in L^2_{\dbF^\a}(s,\infty;\dbR^n)^\perp\times  L_{\dbF}^2(s,\infty;\dbR^n)\times M_{\dbF^\a_-}^2(s,\infty;\dbR^n)^\perp$ and  $(\pi^*_{2,\d},\nu^*_{2,\d})\in L^2_{\dbF^\a}( s,\infty;\dbR^n) \times M_{\dbF^\a_-}^2(s,\infty;\dbR^n) $ satisfying the following infinite horizon BSDE on $[s,\i)$
\begin{align*}\begin{cases}&\!\!\!\!\!\! d\pi^*_{1,\d}(t)-\eta^*_{\d}(t)dW(t)-\nu^*_{1,\d}dM(t) \\
&\!\!\!\!\!\!\qq + \[(A_1^{\widehat\Th_1+\Th^*_{1,\d}})^\top\pi^*_{1,\d}+(C_1^{\widehat\Th_1+\Th^*_{1,\d}})^\top (P^*_{1,\d}\sigma_1+\Pi[\eta^*_{\d}])\!+\!P^*_{1,\d}b_1+q_1+\!(\widehat\Th_1+\Th^*_{1,\d})^\top r_1\]dt=0,\\	
&\!\!\!\!\!\! d\pi^*_{2,\d}(t)\!-\!\nu^*_{2,\d}dM(t)\\
&\!\!\!\!\!\!\qq+\[ (A_2^{\widehat\Th_2+\Th^*_{2,\d}})^\top\pi^*_{1,\d}+(C_2^{\widehat\Th_2+\Th^*_{2,\d}})^\top (P^*_{1,\d}\sigma_2+\Pi^\perp[\eta^*_{\d}])+P^*_{2,\d} b_2+ q_2+(\widehat\Th_2+\Th^*_{2,\d})^\top r_2\] dt=0,\\
&\ds\!\!\!\!\!\!\lim_{t\rightarrow\infty}\dbE|\pi^*_{i,\d}(t)|^2=0.\end{cases}\end{align*}

{\rm (3)} The  optimal strategy  for  Problem (MF-LQ)$_{\d}^{\infty}$ is%
\begin{align*}
u^*_{i,\d}(\cd) =(\widehat\Th_i+\Th^*_{i,\d})X_i-( R_i+D_i^\top P^*_{1,\d}D_i+\d I)^{-1} \[B_i^\top\pi^*_{i,\d}\!+\!D_i^\top\Pi_i[\eta^*_{\d}]\!+\!D_i^\top P^*_{1,\d}\sigma_i\!+\!r_i\],\q i=1,2,
\end{align*}
which is asymptotic optimal for  Problem (MF-LQ)$^{\infty}$.
\end{proposition}
\begin{proof} Now suppose $(\Th_1,\Th_2)\in \BS[A_1,A_2, C_1,C_2; B_1,B_2,D_1,D_2].$ We only need to prove (1) and the proof for (2) and (3) are straightforward. Replacing $A_i,C_i,Q_i,S_i$ by $A^{\widehat\Th_i}_i,C^{\widehat\Th_i}_i,Q^{\widehat\Th_i}_i,S^{\widehat\Th_i}_i$ respectively in \eqref{BSAREPPP}, we have
\begin{align*}
&0=\L[P^*_{i,\d}]+P^*_{i,\d}A^{\widehat\Th_i}_i+(A^{\widehat\Th_i}_i)^\top P^*_{i,\d}+(C_i^{\widehat\Th_i})^\top P^*_{1,\d}C^{\widehat\Th_i}_i +Q^{\widehat\Th_i}_i
\\
&\qq-(B_i^\top P^*_{i,\d}+ D_i^\top P^*_{1,\d} C^{\widehat\Th_i}_i+ S^{\widehat\Th_i}_i) ^\top( R_i+D_i^\top P^*_{1,\d}D_i+\d I)^{-1}(B_i^\top P^*_{i,\d}+ D_i^\top P^*_{1,\d} C^{\widehat\Th_i}_i+ S^{\widehat\Th_i}_i)
\\
&\q=\L[P^*_{i,\d}]+P^*_{i,\d}A_i+A_i^\top P^*_{i,\d}+C_i^\top P^*_{1,\d}C_i +Q_i
\\
&\qq-(B_i^\top P^*_{i,\d}+D_i^\top P^*_{1,\d}C_i+ S_i)^\top( R_i+D_i^\top P^*_{1,\d}D_i+\d I)^{-1}(B_i^\top P^*_{i,\d}+D_i^\top P^*_{1,\d}C_i+ S_i)
\\
&\qq+\d\widehat\Th_i^\top(R_i+D_i^\top P^*_{1,\d}D_i+\d I)^{-1}(B_i^\top P^*_{i,\d}+D_i^\top P^*_{1,\d}C_i+ S_i)
\\
&\qq+\d(B_i^\top P^*_{i,\d}+D_i^\top P^*_{1,\d} C_i+S_i)^\top(R_i+D_i^\top P^*_{1,\d}D_i+\d I)^{-1}\widehat\Th_i
\\
&\qq+\d\widehat\Th_i^\top(R_i+D_i^\top P^*_{1,\d}D_i)(R_i+D_i^\top P^*_{1,\d}D_i+\d I)^{-1}\widehat\Th_i.
\end{align*}
The optimal $\Th^*_{i,\d}$ satisfies
\begin{align*}
&\Th^*_{i,\d}=-( R_i+D_i^\top P^*_{1,\d}D_i+\d I)^{-1}(B_i^\top P^*_{i,\d}+ D_i^\top P^*_{1,\d} C^{\widehat\Th_i}_i+ S^{\widehat\Th_i}_i)
\\
&\q=-( R_i+D_i^\top P^*_{1,\d}D_i+\d I)^{-1}(B_i^\top P^*_{i,\d}+ D_i^\top P^*_{1,\d} C_i+ S_i)
\\
&\qq-( R_i+D_i^\top P^*_{1,\d}D_i+\d I)^{-1}( R_i+D_i^\top P^*_{1,\d}D_i)\widehat\Th_i,
\end{align*}
for $i=1,2$. The proof is complete.
\end{proof}

Before we finish this section, we emphasize that the assumption that    Problem (MF-LQ)$^{\infty}$ being finite is also necessary for the  existence of a sequence asymptotic optimal strategies. Now, a natural question to ask here is whether we can take the limit for $P^*_{i,\d}$ and $\Th^*_{i,\d}$ in \eqref{BSAREPPP} and \eqref{optimalThhomo} to find an optimal control. We see that $P^*_{i,\d}$ is decreasing and bounded from below, and thus admits a limit when $\d\rightarrow0^+.$ While the situation for $\Th^*_{i,\d}$ is totally different. Let us see the following example.
\begin{example} Consider the following one-dimensional SDE
	$$dX(t)=(-X(t)+\dbE u(t))dt+ dW(t).$$
	Then by orthogonal decomposition, we have
\begin{align*}dX_1(t)=-X_1(t)dt+dW(t),\qq d X_2(t)=(-X_2(t)+u_2(t))dt
\end{align*}
 with a cost functional
\begin{align*}J^\infty(s,x;u_1(\cd),u_2(\cd))=\dbE\int_s^\infty |X_2(t)|^2dt.
\end{align*}
We see that $(0,0)$ is a stabilizer. It is obvious that $J^\infty$ is bounded from below and the optimization problem is finite. All the assumptions are fulfilled. While the  detailed calculation yields that
\begin{align*}P^*_{2,\d}=\sqrt{\d^2+\d}-\d\text{ and }\Th^*_{2,\d}=-\d^{-1}(\sqrt{\d^2+\d}-\d).
\end{align*}
Obviously we have $\Th^*_{2,\d}\rightarrow-\infty,$ as $\delta\rightarrow 0^+$. The key reason here is that the optimal control  does not exist and
\begin{align*}\dbE\int_s^\infty|u^*_{2,\d}(t)|^2dt\rightarrow\infty\text{ as }\d\rightarrow 0^+.
\end{align*}
\end{example}

\section{Solvability}\label{sec:sol}

In this section, we will consider the solvability of Problem (MF-LQ)$^{\infty}$. Our main result is to show an equivalent characterization for  the open-loop and closed-loop solvability. The following is our main result.

\begin{theorem} \label{mainthsolv} \sl Under  {\bf (A2)}, the followings are equivalent.

{\rm (i)}. Problem (MF-LQ)$^{\infty}$ is open-loop solvable.

{\rm (ii)}.  Problem (MF-LQ)$^{\infty}$ is closed-loop solvable.

{\rm (iii)}.
The following algebraic equation
\begin{align}\label{BSAREPPP222}\begin{cases}& \!\!\!\!\!\!\L[P^*_{i}]+P^*_{i}A_i+A_i^\top P^*_{i}+C_i^\top P^*_{1}C_i +Q_i
\\
& \!\!\!\!\!\!\qq-(B_i^\top P^*_{i}+ D_i^\top P^*_{1} C_i+ S_i) ^\top( R_i+D_i^\top P^*_{1}D_i)^{\dagger}(B_i^\top P^*_{i}+ D_i^\top P^*_{1} C_i+ S_i)=0,
\\
& \!\!\!\!\!\!R_i+D_i^\top P^*_{1}D_i\ges 0,\qq \sR(B_i^\top P^*_{i}+ D_i^\top P^*_{1} C_i+ S_i)\subset\sR(R_i+D_i^\top P^*_{1}D_i),
\end{cases}\end{align}
admits a solution $(P^*_1,P^*_2):\cM\mapsto\dbS^n\times \dbS^n$  such that $(\Th^*_{1},\Th^*_2)\in\BS[A_1,A_2, C_1,C_2; B_1,B_2,D_1,D_2]$,
where
\begin{align}\label{barThlimit}&\Th^*_{i}=-( R_i+D_i^\top P^*_{1}D_i)^{\dagger}(B_i^\top P^*_{i}+ D_i^\top P^*_{1} C_i+ S_i)+[I-( R_i+D_i^\top P^*_{1}D_i)^\dagger( R_i+D_i^\top P^*_{1}D_i)] \Xi _i
\end{align}
for some $\Xi _i(\cdot):\cM\mapsto \dbR^{m\times n}$.
	
Let   $(\pi^*_{1},\eta^*, \nu^*_{1})\in L^2_{\dbF^\a}(s,\infty;\dbR^n)^\perp\times  L_{\dbF}^2(s,\infty;\dbR^n)\times M_{\dbF^\a_-}^2(s,\infty;\dbR^n)^\perp$ and  $(\pi^*_{2},\nu^*_{2})\in L^2_{\dbF^\a}( s,\infty;\dbR^n) \times M_{\dbF^\a_-}^2(s,\infty;\dbR^n) $ be the solution to the following BSDE on $[s,\i)$
\bel{eqpiooo02222}\left\{\ba{ll}
\ns\ds\!\!\! d\pi^*_{1}(t)-\eta^*(t)dW(t)-\nu^*_{1}\circ dM(t) \\
\ns\ds\q + \[(A_1^{\Th^*_{1}})^\top\pi^*_{1}\!+\!(C_1^{\Th^*_{1}})^\top (P^*_{1}\sigma_1\!+\!\Pi_1[\eta^*])\!+\!P^*_{1}b_1\!+\!q_1\!+\!{\Th^*_{1}}^\top r_1\]dt=0,\\
\ns\ds\!\!\!\! d\pi^*_{2}(t)\!-\!\nu^*_{2}\circ dM(t) \!+\! \[ (A_2^{\Th^*_{2}})^\top\pi^*_{1}\!+\!(C_2^{\Th^*_{2}})^\top (P^*_{1}\sigma_2\!+\!\Pi_2[\eta^*])\!+\!P^*_{2} b_2\!+\!q_2\!+\!{\Th^*_{2}}^\top r_2\] dt=0,\\
\ns\ds\!\!\!\! \lim_{t\rightarrow\infty}\dbE|\pi^*_{i}(t)|^2=0.\ea\right.\eel
It follows that
\begin{align}\label{linearcondition}B_i^\top\pi^*_{i}\!+\!D_i^\top\Pi_i[\eta^*]\!+\!D_i^\top P^*_{1}\sigma_i\!+\!r_i\in \sR( R_i+D_i^\top P^*_{1}D_i).
\end{align}
\end{theorem}

\begin{proof}  In the proof, we will assume {\bf (A2)'} instead of {\bf (A2)} for the sake of convenience.  The proof below can be easily extended to the case when {\bf (A2)} is assumed provided Proposition \ref{proA2}.

(ii)$\Rightarrow$(i): Trivial.

(iii)$\Rightarrow$(ii): The existence of the solution to \eqref{eqpiooo02222} follows from Thereom 5.3 in \cite{Mei-Wei-Yong-2025}. Therefore it suffices to prove the feedback strategy
\begin{align}\label{opt u} u^*_i(t)=\Th^*_i(\a(t))X_i(t)-(R_i+D_i^\top P^*_{1}D_i)^\dagger (B_i^\top\pi^*_{i}\!+\!D_i^\top\Pi_i[\eta^*]\!+\!D_i^\top P^*_{1}\sigma_i\!+\!r_i)
\end{align}
gives an optimal control for  Problem (MF-LQ)$^{\infty}$. Applying It\^o's formula  to $\sum_{i=1}^2\lan  P^*_{i}(\a(t))X_i(t),X_i(t)\ran+2\lan \pi_i(t),X_i(t)\ran $ and completing the square,  one has
\begin{align*}
&J^\infty(s,\iota,x_1\oplus x_2;u_1(\cd)\oplus u_2(\cd))
\\
&\q=  \sum_{i=1}^2\dbE\[\lan P^*_{i}(\iota)x_i,x_i \ran+2\lan  \pi^*_{i}(s)x_i,x_i\ran+\int_s^\i(\lan P^*_{i}\si_i,\si_i\ran+2\lan \pi^*_{i} ,  b_i\ran+2\lan \Pi_i[\eta^*], \si_i \ran )dt
\\
&\qq+  \int_s^\i \( (\lan R_i+D_i^\top P^*_{1}D_i)(u_i-\Th^*_{i} X_i),u_i-\Th^*_{i}X_i  \ran-2\lan B_i^\top \pi^*_{i}+D_i^\top \Pi_i[\eta^*]+D_i^\top P^*_{1}\si_i+r_i, u_i-\Th^*_{i}  X_i\ran\)dt\Big]
\\
&\q \geq  \sum_{i=1}^2\dbE\[\lan P^*_{i}(\iota)x_i,x_i \ran+2\lan \pi^*_{i}(s)x_i,x_i\ran+  \int_s^\i(\lan P^*_{i}\si_i,\si_i\ran  +2\lan \pi^*_{i} ,  b_i\ran+2\lan \eta^*, \si \ran) dt
\\
&\qq- \dbE\int_s^\i \lan(B_i^\top \pi^*_{i}+D_i ^\top P^*_{1}\si_i+D_i^\top\Pi_i[\eta^*]+r_i), (R_i+D_i^\top P^*_{1}D_i )^\dagger(B_i^\top \pi^*_{i}+D_i ^\top P^*_{1}\si_i+D_i^\top\Pi_i[\eta^*]+r_i)\ran dt\Big].
\end{align*}
For any closed-loop strategy $u_i(\cd)=\Th_i(\a(\cd)) X_i(\cd)+v_i(\cd)$, 	the equality holds if and only if $\Th_i=\Th^*_i$ and $v_i=-(R_i+D_i^\top P^*_{1}D_i)^\dagger (B_i^\top\pi^*_{i}\!+\!D_i^\top\Pi_i[\eta^*]\!+\!D_i^\top P^*_{1}\sigma_i\!+\!r_i)$ hold. This says $u^*_i$ defined in \eqref{opt u} is the optimal closed-loop strategy.\ss
	
(i)$\Rightarrow$(iii):
Because  Problem (MF-LQ)$^{\infty}$ is solvable,  Problem (MF-LQ)$_0^{\infty}$ is finite. Theorem \ref{P-i-d-opti} yields that the optimal strategy for  Problem (MF-LQ)$_{0,\d}^{\infty}$ can be represented as
$u^*_{i,\d}(t)=\Th^*_{i,\d}(\a(t))X_i(t).$
Now we aim to show that $\Th^*_{i,\d}(\cd)$ is bounded to take the limit as $\d\rightarrow 0^+$.

Because  Problem (MF-LQ)$_0^{\infty}$ is open-loop solvable, by (iii) in Proposition \ref{solvab}, let $ u^*(\cd)=u_1^*(\cd)\oplus u_2^*(\cd)\in\sU[s,\infty)$ be the optimal control with $u^*(\cd)=[U^*(\iota,\cd) x](\cd)$. Then we have
\begin{align*}&V^{0,\infty}(s,\iota,x_1\oplus x_2)+\d\dbE\int_s^\infty\sum_{i=1}^2| u^*_{i}(t)|^2dt= J_\d^{0,\infty}(s,\iota,x_1\oplus x_2;u^*_{1}(\cd)\oplus u^*_{2}(\cd))
\\
&\q\ges V_\d^{0,\infty}(s,\iota,x_1\oplus x_2)=J^{0,\infty}(s,\iota,x_1\oplus x_2;u^*_{1,\d}(\cd)\oplus u^*_{2,\d}(\cd))+\d\dbE\int_s^\infty\sum_{i=1}^2|u^*_{i,\d}(t)|^2dt
\\
&\q\ges V^{0,\infty}(s,\iota,x_1\oplus x_2)+\d\dbE\int_s^\infty\sum_{i=1}^2|u^*_{i,\d}(t)|^2dt.
\end{align*}
This says that
\begin{align}\label{bondeduid}
\dbE\int_s^\infty\sum_{i=1}^2|u^*_{i,\d}(t)|^2dt\les \dbE\int_s^\infty\sum_{i=1}^2| u^*_{i}(t)|^2dt.
\end{align}
Note that $u^*_{i,\d}(t)=\Th^*_{i,\d}(\a(t))X_i(t)$,  we have
$$u^*_{2,\d}(t)=\bar \Psi^{s,\iota}_{2,\d}(t)x_2,$$ where $\bar \Psi^{s,\iota}_{2,\d}$ solves
$$d\bar\Psi^{s,\iota}_{2,\d}(t)=(A_2+B_2\Th^*_{2,\d})(\a(t))\bar\Psi^{s,\iota}_{2,\d}(t)dt,\q t\in[s,\i).$$
Write $$\bar\Phi_{2,\d}(\iota):=\dbE\[\int_s^\infty (\bar \Psi_{2,\d}^{s,\iota})^\top(\a(t)){\Th^*_{2,\d}}^\top(\a(t))\Th^*_{2,\d}(\a(t))\bar  \Psi_{2,\d}^{s,\iota}(\a(t))dt\Big|\a(s)=\iota\].$$
Then it follows that
$$\lan x_2,  \bar\Phi_{2,\d}(\iota) x_2\ran=\int_s^\infty|u^*_{2,\d}(t)|^2dt\leq \dbE\int_s^\infty\sum_{i=1}^2| u^*_{i}(t)|^2dt\leq \sum_{i=1}^2\int_s^\infty\lan \dbE |[U^*(s,\iota) x](t)|^2\ran dt.$$
Note that $U^*(s,\iota)$ is a linear operator over $x$,  $\bar \Phi_{2,\d}(\iota)$ is positive-definite and bounded  from above.
Similar to the proof of Proposition 3.3. in \cite{Mei-Wei-Yong-2025}, by taking $L_2={\Th^*_{2,\d}}^\top\Th^*_{2,\d}$,
 one has
$$\L[\bar\Phi_{2,\d}] +\bar\Phi_{2,\d}(A_2+B_2\Th^*_{2,\d})+(A_2+B_2\Th^*_{2,\d})^\top \bar\Phi_{2,\d}+{\Th^*_{2,\d}}^\top\Th^*_{2,\d}=0.$$
This concludes that $\Th^*_{2,\d}$ is  bounded necessarily. Let $x_2=0$ and repeat the above proof for $u^*_{1,\d}(\cd)$. We can prove  $\Th^*_{2,\d}$ is uniformly bounded as well.

Due to the boundedness of $(\Th^*_{1,\d},\Th^*_{2,\d})$, taking $\d\rightarrow 0^+$,  there exists a convergent subsequence $\d_k$ such that $(\Th^*_{1,\d_k},\Th^*_{2,\d_k})$ has a limit $(\Th^*_{1},\Th^*_2)$.   At the same time, $P^*_{i,\d}$ converges to $P^*_{i}$ because it is decreasing in $\d$. Consequently, we have
\begin{align*}
&0=\lim_{\d_k\rightarrow 0}\[(R_i+D_i^\top P^*_{1,\d_k} D_i+\d_k I)\Th^*_{i,\d_k}+(B_i^\top P^*_{i,\d_k}+ D_i^\top P^*_{1,\d_k} C_i+ S_i)\]
\\
&=(R_i+D_i^\top P^*_{1} D_i)\Th^*_{i}+(B_i^\top P^*_{i}+ D_i^\top P^*_{1} C_i+ S_i).
\end{align*}
This implies \eqref{barThlimit}. Now let us verify that $(\Th^*_{1},\Th^*_2)\in \BS[A_1,A_2, C_1,C_2; B_1,B_2,D_1,D_2]$.

Define \begin{align*}\begin{cases}
&\!\!\!\!\!\!d\bar\Psi^{s,\iota}_1(t)=(A_1+B_1\Th^*_{1})(\a(t))\bar\Psi^{s,\iota}_1(t)dt+(C_1+D_1\Th^*_{1})(\a(t))\bar\Psi^{s,\iota}_1(t)dW(t),
\\
&\!\!\!\!\!\!d\bar\Psi^{s,\iota}_2(t)=(A_2+B_2\Th^*_{2})(\a(t))\bar\Psi^{s,\iota}_2(t)dt,\q t\in[s,\i),
\\
&\!\!\!\!\!\!\bar\Psi^{s,\iota}_1(s)=\bar\Psi^{s,\iota}_2(s)=I,\q\a(s)=\iota.
\end{cases}
\end{align*}
Fatou's lemma yields that
$$\dbE\int_s^\infty|\Th^*_2(\a(t))\bar\Psi^{s,\iota}_2(t)x_2|^2dt\les\liminf_{\d_k\rightarrow 0} \dbE\int_s^\infty|\Th^*_{2,\d}(\a(t))\bar  \Psi^{s,\iota}_{2,\d_k}(t)x_{2}|^2dt\les \dbE\int_s^\infty|[U^*(\iota,t)x](t)|^2dt.$$
Consider the special case when $x_2=0$. In this case $X^0_2(t)=0$ and   we have $X_1^0(t)=\bar\Psi^{s,\iota}_1(t)x_1$, we have
$$\dbE\int_s^\infty|\Th^*_1(\a(t))\bar\Psi^{s,\iota}_1(t)x_1|^2dt\les\liminf_{\d_k\rightarrow 0} \dbE\int_s^\infty|\Th^*_{1,\d}(\a(t))\bar  \Psi^{s,\iota}_{1,\d_k}(t)x_{1}|^2dt\les \dbE\int_s^\infty|[U^*(\iota,t)x](t)|^2dt.$$
By Proposition 3.3 in \cite{Mei-Wei-Yong-2025}, it follows that	 $(\Th^*_1,\Th^*_2)\in \BS[A_1,A_2, C_1,C_2; B_1,B_2,D_1,D_2]$.
	
Since $(\Th^*_{1},\Th^*_2)\in \BS[A_1,A_2, C_1,C_2; B_1,B_2,D_1,D_2]$, similar to \eqref{eqpiooo0}, \eqref{eqpiooo02222} admits a solution. Now let us verify \eqref{linearcondition}. Applying It\^o's formula on $\sum_{i=1}^2\lan  P^*_{i}(\a(t))X_i(t),X_i(t)\ran+2\lan \pi_i(t),X_i(t)\ran $ and completing the square,  one has
\begin{align}\label{cJu1u2}
&J^\infty(s,\iota,x_1\oplus x_2;u_1(\cd)\oplus u_2(\cd))
 \\
&\!\! =  \sum_{i=1}^2\dbE\[\lan P^*_{i}(\iota)x_i,x_i \ran+2\lan \pi^*_{i}(s)x_i,x_i\ran+\int_s^\i\lan P^*_{i}\si_i,\si_i\ran+2\lan \pi^*_{i} ,  b_i\ran+2\lan \Pi_i[\eta^*], \si_i \ran dt\]\nonumber
\\
& +  \int_s^\i \( (\lan R_i+D_i^\top P^*_{1}D_i)(u_i-\Th^*_{i} X_i),u_i-\Th^*_{i}X_i  \ran  \nonumber \\
&\qq \qq -2\lan B_i^\top \pi^*_{i}+D_i^\top \Pi_i[\eta^*]+D_i^\top P^*_{i}\si_i+r_i, u_i-\Th^*_{i}  X_i\ran\)dt.\!\! \nonumber
\end{align}
To guarantee that $J^\infty(s,\iota,x_1\oplus x_2;u_1(\cd)\oplus u_2(\cd))$ is bounded from below, it is necessary that \eqref{linearcondition} holds. The proof is complete.
\end{proof}

\begin{remark}\rm We remark that the optimal control obtained in Theorem \ref{mainthsolv} is independent of the choice of $(\widehat\Th_1,\widehat \Th_2)\in \BS[A_1,A_2, C_1,C_2; B_1,B_2,D_1,D_2]$, different from the asymptotic optimal sequence constructed in Proposition  \ref{proA2}. This is natural because the asymptotic optimal sequence of controls might not be unique.
\end{remark}

\section{Concluding Remarks}\label{sec:con}

In this paper, we have studied a stochastic LQ problem in an infinite horizon and in a switching environment. It is noted that the cost functional is not assumed to have positive-definite weights necessarily. When the problem is merely finite, a sequence of asymptotic optimal strategies is obtained in forms of the solutions to AREs and BSDEs. Moreover, when the problem is solvable, we derive an equivalent characterization for the open-loop and closed-loop solvablities. Further on this topic, the LQ control problem under an ergodic-type cost would be of interest for future study.

%  \section*{Declarations}
% \no \textbf{Conflict of interest} The authors have no relevant financial or non-financial interests to disclose.
%
%

	\end{document}